\documentclass[12pt]{amsart}
\usepackage{amssymb,amsmath,amsfonts,amsthm,nomencl,mathrsfs} 
\usepackage[arrow, matrix, curve]{xy}
\usepackage[latin1]{inputenc}
\usepackage{a4wide}

\newcommand{\IR}{\mathbb{R}}

\newcommand{\question}[1]{\leavevmode{\marginpar{\tiny%
$\hbox to 0mm{\hspace*{-0.5mm}$\leftarrow$\hss}%
\vcenter{\vrule depth 0.1mm height 0.1mm width \the\marginparwidth}%
\hbox to 0mm{\hss$\rightarrow$\hspace*{-0.5mm}}$\\\relax\raggedright #1}}}

\newcommand{\ICC}{\mathsf{C}^{\infty}}

\newcommand{\IFF}{\mathscr{F}}
\newcommand{\IAA}{\mathscr{A}}

\newcommand{\IL}{\mathsf{L}}

\newcommand{\IN}{\mathbb{N}}

\newcommand{\IP}{\mathbb{P}}

\newcommand{\Id}{{\rm d}}

\newcommand{\f}{\frac}
\newcommand{\nn}{\nonumber}

\theoremstyle{plain}            
\newtheorem{theorem}{theorem}[section]
\newtheorem{Lemma}[theorem]{Lemma}
\newtheorem{Corollary}[theorem]{Corollary}
\newtheorem{Theorem}[theorem]{Theorem}
\newtheorem{Proposition}[theorem]{Proposition}
\newtheorem{Propandef}[theorem]{Proposition and definition}

\theoremstyle{definition}       
\newtheorem{Definition}[theorem]{Definition}

\newtheorem{Remark}[theorem]{Remark}

\allowdisplaybreaks[1]

\begin{document}

\begin{titlepage}
\title[]{On the semimartingale property of Brownian bridges on complete manifolds}

\author[B. G\"uneysu]{Batu G\"uneysu}

\end{titlepage}

\maketitle 

\begin{abstract}
I prove that every adapted Brownian bridge on a geodesically complete connected Riemannian manifold is a semimartingale including its terminal time, without any further assumptions on the geometry. In particular, it follows that every such process can be horizontally lifted to a smooth principal fiber bundle with connection, including its terminal time. The proof is based on a localized Hamilton-type gradient estimate by Arnaudon/Thalmaier.
\end{abstract}

%

\section{Introduction}

Given $x,y\in \IR^m$, $T>0$, let $\Omega_T(\IR^m)$ stand for the Wiener space of continuous paths $\omega:[0,T]\to \IR^m$. We denote with $\IP^{x,T}$ the usual Euclidean Wiener measure (=Brownian motion measure) and with $\IP^{x,y,T}$ the usual Euclidean pinned Wiener measure (=Brownian bridge measure) on $\Omega_T(\IR^m)$ with its Borel-sigma-algebra $\IFF^T$. Then with $X$ the coordinate process on $\Omega_T(\IR^m)$ and $\IFF^T_*=(\IFF^T_t)_{t\in [0,T]}$ the filtration of $\IFF^T$ that is generated by $X$, the following important fact is well-known to hold true:
\begin{itemize}
\item[(SM)] \emph{$(X_t)_{t\in [0,T]}$ is a continuous semimartingale with respect to $(\IP^{x,y,T}, \IFF^T_*)$}. 
\end{itemize}

Let us point out here that, as for all $0<t<T$ one has $\IP^{x,y,T}|_{\IFF^T_t}\sim  \IP^{x,T}|_{\IFF^T_t}$, the property (SM) only becomes nontrivial at $t=T$. Furthermore, the importance of (SM) is already clear at a very fundamental level: Continuous disintegrations of probabilistic formulae for covariant Schrödinger semigroups \cite{G1,G2} clearly require such a result. We refer the reader to \cite{G1} for such a continuous disintegration in the Euclidean case. \vspace{1mm}

In this paper we will be concerned with the validity of the semimartingale property (SM) on noncompact Riemannian manifolds. To this end, we start by recalling that given a connected Riemannian manifold $M$ of dimension $m$, the corresponding Riemannian data $\Omega_T(M)$, $\IP^{x,y,T}$, $\IP^{x,T}$, $X$, $\IFF^T$, $\IFF^T_*$, as well as the question whether one has (SM) or not still make sense: One just has to take the minimal positive heat kernel $p(t,x_1,x_2)$ everywhere in the definition of the underlying measures, replacing the Euclidean heat kernel $t^{-m/2}\mathrm{e}^{-|x_1-x_2|^2/(4t)}$ (cf. Definition \ref{wiener} below), and to note that $X$ is a continuous $M$-valued semimartingale, if and only if $f(X)$ is a real-valued one, for all smooth functions $f:M\to \IR$. In fact (SM), has been established quite some time ago (1984) on \emph{compact $M$'s} by Bismut \cite{bismut}, who used the resulting \lq\lq{}covariant\rq\rq{} continuous disintegration in his proof of the Atiyah-Singer index theorem \cite{bisat} (the reader may also wish to consult \cite{driver} and \cite{Hsu}).\\

Concerning (SM) for noncompact $M\rq{}s$, we point out that this property has been stated in \cite{aida} under very restrictive geometric assumptions, such as a bounded Ricci curvature plus a positive injectivity radius, and indeed there seems to be a widely spread belief that (SM) requires some global curvature bounds in order to hold true. The reason for this might be that if one follows the typical proofs from the compact case too closely, it is tempting to believe that one needs to establish the integrability
\begin{align}\label{intro3}
\mathbb{E}^{x,y,T}\left[\int^T_0 \big|\Id \log p(T-t,\bullet,y)(X_t)\big| \Id t\right]<\infty,
\end{align}
which in the compact case is proved using a global gradient estimate of the form
\begin{align}\label{intro4}
|\Id \log p( t,\bullet,x_1)(x_2)|\leq C_{T}\big(t^{-1/2} +t^{-1}\Id(x_1,x_2)\big) \quad\text{ for all $x_1,x_2\in M$, $0<t\leq T$,}
\end{align}
an inequality that certainly requires global curvature bounds. We point out here that problems like (\ref{intro3}) arise naturally in this context, as under $(\IP^{\bullet,y,T}, \IFF^T_*)$ the process $X|_{[0,T)}$ is a diffusion which is generated by the time-dependent differential operator
\begin{align*}
[\IAA^y_t f](z)= (1/2)\Delta f(z) - \big(\Id \log p(T-t, \bullet ,y)(z), \Id f\big),\quad \text{$f\in\ICC(M)$, $0<t<T$, $z\in M$,}
\end{align*}
where $\Delta$ denotes the Laplace-Beltrami operator, and where $\Id F$ stands for the differential of a function $F:M\to \IR$. In particular, as $p(t,x_1,x_2)$ becomes singular near $t=0$, it is clear that \emph{some} upper bound on $|\Id \log p( t,\bullet,x_1)(x_2)|$ for small $t$ has to be established in any case. It is also instructive to note that in the Euclidean case one has
\begin{align*}
[\IAA^y_t f](z)= (1/2)\Delta f(z) + (1/2)(T-t)^{-1} \sum^m_{i=1} (y^i-z^i) \partial_i f(z),\quad z=(z^1,\dots,z^m)\in\IR^m.
\end{align*}

\emph{In this paper, we are going to prove that (SM) holds true on every geodesically complete connected Riemannian manifold, without any further curvature assumptions.} \vspace{2mm}

In fact, this will be a consequence of our main result Theorem \ref{main} below. Let us continue with some comments on the proof of Theorem \ref{main}, with a view towards the above mentioned technical problems: Firstly, a very simple but nevertheless essential observation is that the semimartingale property of a manifold-valued continuous adapted process $X\rq{}$ can be geometrically localized, in the sense that $X\rq{}$ is a semimartingale if and only if $f(X\rq{})$ is a real-valued one, for all smooth \emph{compactly supported} functions $f$ on the manifold. This fact can be used to deduce that for (SM) to hold it is actually enough to prove the integrability 
\begin{align}\label{intro5}
\mathbb{E}^{x,y,T}\left[\int^T_0 \big|\Id \log p(T-t,\bullet,y)(X_t)\big| \big|\Id f(X_t)\big|\Id t\right]<\infty\quad\text{ for all $f\in\ICC_{\mathrm{cpt}}(M)$}.
\end{align}
Still, one is faced with the problem of establishing a localized version of (\ref{intro4}). In this context, using a highly subtle local parabolic gradient bound by Arnaudon/Thalmaier from 2010 \cite{thal}, we are able to prove the following inequality:\emph{Namely, in Proposition \ref{help} we show that for every $z_0\in M$, $R_0>0$ there exists a constant $C>0$ which depends on the geometry of $M$ in a neighbourhood of $B(z_0,R_0)$, such that for all 
$$
(t,x_1,x_2)\in (0,R_0]\times B(z_0,R_0) \times B(z_0,R_0)
$$
one has }
\begin{align*}
\left|\Id \log p(t,\bullet,x_1)(x_2)\right|\leq C\big(t^{-1/2} +t^{-1}\Id(x_1,x_2)\big).
\end{align*}
Ultimately, we show that the latter localized bound is enough to establish (\ref{intro5}).\vspace{2mm}

Finally, we would like to add that in fact our main result Theorem \ref{main} is more general than (SM) in the following sense: We define (cf. Definition \ref{ada} below) an arbitrary adapted continuous stochastic process 
$$
X^{x,y,T}:[0,T]\times \big(\Omega,\IFF,(\IFF_t)_{t\in [0,T]},\IP\big)\longrightarrow M
$$
to be an adapted Brownian bridge from $x$ to $y$ at the time $T$, if the law of $X^{x,y,T}$ is equal to $\IP^{x,y,T}$ and if $X^{x,y,T}$ has a certain time-inhomogeneous Markoff property. In case the law of  $X^{x,y,T}$ is equal to $\IP^{x,y,T}$, then this Markoff property is shown to be automatically satisfied for $(\IFF_t)_{t\in [0,T]}=(\IFF^{X^{x,y,T}}_t)_{t\in [0,T]}$ (cf. Lemma \ref{shoa} below). In this context, \emph{our main result Theorem \ref{main} states that in fact every adapted Brownian bridge is a continuous semimartingale.} In particular, this result entails that (cf. Corollary \ref{sjoa} below) every adapted Brownian bridge on a geodesically complete Riemannian manifold can be horizontally lifted to principal bundles that are equipped with a connection.


\vspace{4mm}

\textbf{Acknowledgements:} I would like to thank A. Thalmaier for a helpful discussion concerning \cite{thal}. It is also a pleasure to thank Shu Shen for a very careful reading of the manuscript. This research has been financially supported by the SFB 647: Raum-Zeit-Materie. 

\section{Main results}

Let $M\equiv (M,g)$ be a smooth connected Riemannian $m$-manifold, with $\Delta$ its Laplace-Beltrami operator. Let $\Id(x,y)$ denote the geodesic distance, and $B(x,r)$ corresponding open balls. We denote with $p(t,x,y)$, $t>0$, $x,y\in M$, the minimal nonnegative heat kernel on $M$, that is, for each fixed $y$, the function $p(\bullet,\bullet,y)$ is the pointwise minimal nonnegative smooth fundamental solution of the following heat equation in $(0,\infty)\times M$, 
$$
(\partial /\partial t-(1/2)\Delta)p(\bullet,\bullet,y)=0,\>\>p(t,\bullet,y)\to \delta_y\>\>\text{ as $t\to 0+$.}
$$
It follows that $(t,x,y)\mapsto p(t,x,y)$ is jointly smooth, and the connectedness of $M$ implies in fact the positivity $p(t,x,y)>0$. With $\Id\mu$ the Riemannian volume measure we define
$$
P_tf(z):=\int  p(t,z,w)f(w)\Id\mu (w),\>\> \text{ for every }\>\> f\in\bigcup_{q\in [1,\infty]}\IL^q(M), \>z\in M,\> t>0.
$$
Then $(t,z)\to P_tf(z)$ is smooth in $(0,\infty)\times M$. \vspace{2mm}

Given $T>0$ we denote with $\Omega_T(M)$ the Wiener space of continuous paths $\omega:[0,T]\to M$. We give the latter the topology of uniform convergence. Let $\IFF^T$ denote the Borel-sigma algebra on $\Omega_T(M)$, and let $\IFF^T_*:=(\IFF^T_t)_{t\in [0,T]}$ denote the filtration of $\IFF^T$ which is generated by the underlying canonical coordinate process. Note here that $\IFF^T=\IFF^T_T$. The following result is well-known (cf. \cite{baer} for a detailed proof):

\begin{Propandef}\label{wiener} 1. For every $x_0\in M$, \emph{the Wiener measure $\IP^{x_0,T}$ from $x_0$ with terminal time $T$} is defined to be the unique sub-probability measure on $(\Omega_T(M),\IFF^T)$ which satisfies
\begin{align*}
&\IP^{x_0,T}\{\omega\in \Omega_T(M):\omega(t_1)\in A_1,\dots,\omega(t_n)\in A_n\} \\
&=\int_{A_1}\cdots\int_{A_n} p(\delta_0 ,x_0,x_1) \cdots p(\delta_{n-1} ,x_{n-1},x_n) \Id \mu(x_1)\cdots \Id\mu(x_n)
\end{align*}
for all $n\in\IN_{\geq 1}$, all partitions $0=t_0<t_1< \dots<t_{n-1}<t_n=T$ and all Borel sets $A_1,\dots,A_n\subset M$, where $\delta_j:=t_{j}-t_{j-1}$.\\
2. For every $x_0,y_0\in M$, the \emph{pinned Wiener measure $\IP^{x_0,y_0,T}$ from $x_0$ to $y_0$ with terminal time $T$} is defined to be the unique probability measure on $(\Omega_T(M),\IFF^T)$ which satisfies
\begin{align*}
 \IP^{x_0,y_0,T}(A)=\f{1}{p(T,x_0,y_0)}\int_A p(1-t,\omega(t),y_0)\Id \IP^{x_0,T}(\omega)\>\text{ for all $0\leq t<T$, and all $A\in \IFF^T_t$.}
\end{align*}
\end{Propandef}

It has been shown by E. Hsu \cite{hsu2} that the pinnded Wiener measure satisfies a natural large deviation principle under geodesic completeness.\\
The following well-known facts follow straightforwardly from the definitions and will be used repeatedly in the sequel:

\begin{Remark} 1. For every $x,y \in M$ one has
$$
\IP^{x,T}\{\omega \in\Omega_T(M): \omega(0)=x\}=1=\IP^{x,y,T}\{\omega\in  \Omega_T(M): \omega(0)=x,\omega(T)=y\},
$$
as it should be. Furthermore $\IP^{x,T}(\Omega_T(M))\leq 1$, whereas $\IP^{x,y,T}(\Omega_T(M))= 1$, which reflects the fact that \lq\lq{}paths with explosion time\rq\rq{} that are initially and terminally pinned on $M$ cannot explode.\\
2. One has the following \emph{time reversal symmetry} of the pinned Wiener measure: The pushforward of $\IP^{x,y,T}$ with respect to the $\IFF^T/\IFF^T$ measurable map $\Omega_T(M)\to \Omega_T(M)$ given by $\omega\mapsto \omega(T-\bullet)$ is precisely $\IP^{y,x,T}$.
\end{Remark}

Now we can give:

\begin{Definition}\label{ada} Let $(\Omega,\IFF,\IP)$ be a probability space, and let 
$$
X^{x,y,T}: [0,T]\times\Omega \longrightarrow M
$$
be a continuous process\footnote{which is thus automatically jointly measurable; furthermore, in the sequel we will identify indistinguishable process.}. Then $X^{x,y,T}$ is called \emph{a Brownian bridge from $x$ to $y$ with terminal time $T$} on $M$, if $(\widetilde{X^{x,y,T}})_*\IP=\IP^{x,y,T}$, where 
$$
\widetilde{X^{x,y,T}}: \Omega\longrightarrow \Omega_T(M),\>\>\widetilde{X^{x,y,T}}(\omega):=X^{x,y,T}_{\bullet}(\omega)
$$
denotes the induced $\IFF/\IFF^T$ measurable map. \\
In this situation, if in addition $X^{x,y,T}$ is adapted to a filtration $\IFF_*:=(\IFF_t)_{t\in [0,T]}$ of $\IFF$, then $X^{x,y,T}$ is called an \emph{$\IFF_*$-(adapted) Brownian bridge}, if in addition there holds the following time-inhomogeneous Markoff property: For all numbers $0\leq S<T$, all bounded $\IFF_{S}$-measurable $\Phi:\Omega\to \IR$, and all bounded continuous (thus $\IFF^T$-measurable) $\Psi:\Omega_T(M)\to\IR$ one has
$$
\mathbb{E}\left[ \Phi\cdot\Psi(X^{x,y,T}_{\min(S+\bullet,T)})\right]=\mathbb{E}\left[\Phi\cdot\int \Psi\Big(\omega(\min(\bullet,T-S)\Big) \Id\IP^{X^{x,y,T},y,T-S}(\omega)\right].
$$
\end{Definition}

The particular form of the latter Markoff property is motivated by the fact that every Brownian bridge satisfies this Markoff property with respect to its own filtration:

\begin{Lemma}\label{shoa} Let $(\Omega,\IFF,\IP)$ be a probability space, and let 
$$
X^{x,y,T}: [0,T]\times\Omega \longrightarrow M
$$
be a Brownian bridge from $x$ to $y$ with terminal time $T$. Then, with $\IFF^{X^{x,y,T}}_*$ the filtration generated by $X^{x,y,T}$, one has the following time-inhomogeneous Markoff property: For all numbers $0\leq S<T$, all bounded $\IFF^{X^{x,y,T}}_{S}$-measurable $\Phi:\Omega\to \IR$, and all bounded continuous (thus $\IFF^T$-measurable) $\Psi:\Omega_T(M)\to\IR$ one has
$$
\mathbb{E}\left[ \Phi\cdot\Psi(X^{x,y,T}_{\min(S+\bullet,T)})\right]=\mathbb{E}\left[\Phi\cdot\int \Psi\Big(\omega\big(\min(\bullet,T-S)\big)\Big) \Id\IP^{X^{x,y,T},y,T-S}(\omega)\right],
$$
in other words, $X^{x,y,T}$ is an $\IFF^{X^{x,y,T}}_*$-Brownian bridge.
\end{Lemma}

\begin{proof} As by the Doob-Dynkin Lemma one can write $\Phi=F(X^{x,y,T})$ for some $\IFF^T_S$-measurable function $F: \Omega_T(M)\to \IR$, we can and we will assume that the underlying filtered probability space is given by $(\Omega_T(M),\IFF^T,\IFF^T_*,\IP^{x,y,T})$ with its the coordinate process $X$. Now we can follow the proof of the Euclidean case which is given in Sznitman\rq{}s book \cite{sznit}, pp. 139/140: For all $0<\delta <T$ we have, with $X$ the coordinate process on $\Omega_T(M)$,
\begin{align*}
&p(T,x,y)\mathbb{E}^{x,y,t}\left[\Phi\cdot\int \Psi\Big(\omega\big(\min(\bullet,T-S-\delta)\big)\Big) \Id\IP^{X^{x,y,T},y,T-S}(\omega)\right]\\
&=\mathbb{E}^{x,T}\left[\Phi\cdot p(T-S, X_S,y)\int \Psi\Big(\omega\big(\min(\bullet,T-S-\delta)\big)\Big) \Id\IP^{X^{x,y,T},y,T-S}(\omega)\right]\\
&=\mathbb{E}^{x,T}\left[\Phi\cdot \int \Psi\Big(\omega\big(\min(\bullet,T-S-\delta)\big)\Big) p(\delta, \omega(T-\delta-S),y) \Id\IP^{X^{x,y,T},T-S}(\omega)\right]\\
&=\mathbb{E}^{x,T}\left[\Phi\cdot \Psi\Big(X_{\min(S+\bullet,T-S-\delta)}\Big) p(\delta, X_{T-\delta},y) \right]\\
&=p(T,x,y)\mathbb{E}^{x,y,T}\left[\Phi\cdot \Psi\Big(X_{\min(S+\bullet,T-S-\delta)}\Big) \right],
\end{align*}
where we have used the defining relation of the pinned wiener measure for the first two equalities, the usual Markoff property of the Wiener measure (in the sense of a sub-probability measure) for the third equality, and once again the defining relation of the pinned wiener measure for the last equality. Finally, taking $\delta \to 0+$ completes the proof using dominated convergence, noting that $\Psi$ is continuous and bounded.
\end{proof}


As we allow filtered probability spaces that need not satisfy the usual assumptions, and as there exist definitions of the term  'semimartingale' (such as 'good integrators') that do not make any sense on such spaces, we add:

\begin{Remark} In the sequel, given $T>0$ and a filtered probability space $(\Omega,\IFF,(\IFF_t)_{t\in [0,T]},\IP)$ we will call a real-valued continuous process
$$
Y : [0,T]\times\Omega \longrightarrow \IR
$$
a \emph{continuous semimartingale w.r.t. $\IFF_*:=(\IFF_t)_{t\in [0,T]}$}, if there exist continuous processes
$$
Y_1,Y_2 : [0,T]\times\Omega \longrightarrow \IR
$$
such that 
\begin{itemize}
\item $Y_1$ is adapted to $\IFF_*$ with paths having a finite variation
\item $Y_2$ is an $\IFF_*$-local martingale
\item $Y=Y_1+Y_2$.
\end{itemize}
We recall further that, following L. Schwartz, a continuous $\IFF_*$-adapted manifold-valued process
$$
X: [0,T]\times\Omega \longrightarrow M
$$
is called a \emph{continuous semimartingale w.r.t. $\IFF_*$}, if for all smooth $f:M\to \IR$ the process
$$
f(X):[0,T]\times\Omega \longrightarrow \IR
$$
is a real-valued continuous semimartingale with respect to $\IFF_*$ in the sense of the former definition. With this definition, it follows that if $X$ as above is a continuous semimartingale w.r.t. $(\Omega,\IFF,(\IFF_t)_{t\in [0,T]},\IP)$, then it is also one w.r.t. the minimal extension of the latter filtration which satisfies the usual assumptions (of completeness and right-continuity). Furthermore, given a continuous $\IFF_*$-adapted $X$ as above, if for every smooth $f:M\to \IR$ one can find a sequence of $\IFF_*$-stopping times $\tau_n:\Omega\to [0,T]$ which announces $T$ in a way that $f(X_{\min(\bullet,\tau_n)})$ is a continuous $\IFF_*$-semimartingale, then $X$ is already a continuous $\IFF_*$-semimartingale.
\end{Remark}

The latter probabilistic localization leads to a simple geometric localization:

\begin{Lemma}\label{lo} Assume we are given $T>0$, a filtered probability space $(\Omega,\IFF,(\IFF_t)_{t\in [0,T]},\IP)$, and a continuous $\IFF_*:=(\IFF_t)_{t\in [0,T]}$-adapted continuous process
$$
X : [0,T]\times\Omega \longrightarrow M
$$
such that for all smooth \emph{compactly supported} $\phi:M\to\IR$ the process $\phi(X)$ is a continuous $\IFF_*$-semimartingale. Then $X$ is a continuous $\IFF_*$-semimartingale.
\end{Lemma}

\begin{proof} Let $f$ be an arbitrary smooth function on $M$, let $(U_n)_{n\in\IN}$ be an open relatively compact exhaustion of $M$. Then, as $X$ is continuous and adapted, and $U_n$ is open, the first exit time $\tau_n$ of $X$ from $U_n$ is a stopping time, for each $n\in\IN$, and $(\tau_n)$ announces $T$. If $\phi_n$ is a smooth compactly supported function $M$ with $\phi_n\equiv 1$ on $\overline{U_n}$, then by assumption $(\phi_n f)(X)$ is a continuous semimartingale which coincides with $f(X_{\min(\bullet,\tau_n)})$.
\end{proof}

Now we can formulate our main result:

\begin{Theorem}\label{main} Assume that $M$ is geodesically complete, let $(\Omega,\IFF,(\IFF_t)_{t\in [0,T]},\IP)$ be a filtered probability space, and let $x,y\in M$, $T>0$. Assume furthermore that $X^{x,y,T}$ is an $\IFF_*:=(\IFF_t)_{t\in [0,T]}$-Brownian bridge from $x$ to $y$ with terminal time $T$. Then for all smooth compactly supported $f:M\to\IR$ one has
$$
\mathbb{E}\left[\int^T_0 \left|\Id \log p^y(T-r,X^{x,y,T}_r)\right|\left| \Id f(X^{x,y,T}_r) \right| \Id r\right]<\infty,
$$
and the real-valued process
\begin{align*}
&X^{x,y,T,f}:[0,T]\times\Omega \longrightarrow \IR,\>\>X^{x,y,T,f}_s:=f(X^{x,y,T}_s)-f(X^{x,y,T}_0)-(1/2)\int^s_0\Delta f(X^{x,y,T}_r)\Id r\\
&\>\>\>-\int^s_0 \big(\Id \log p^y(T-r,X^{x,y,T}_r),\Id f( X^{x,y,T}_r)\big)\Big)\Id r
\end{align*}
is a continuous $\IFF_*$-local martingale. In particular, $X^{x,y,T}$ is a continuous semimartingale with respect to $\IFF_*$.
\end{Theorem}

The following localized heat kernel bounds will play a central role in the proof of Theorem \ref{main}:

\begin{Proposition}\label{help} Let $M$ be geodesically complete.\\
a) For every $z_0\in M$ and every $R>0$ there exist constants $C_j>0$ (which depend on the geometry of $M$ in a neighbourhood of $B(z_0,R)$) such that for all $(t,x,y)\in (0,R]\times B(z_0,R) \times B(z_0,R)$ one has 
\begin{align*}
C_1t^{-m/2}\mathrm{e}^{-C_2\f{\Id(x,y)^2}{t}}\leq p(t,x,y)\leq C_3t^{-m/2}\mathrm{e}^{-C_4\f{\Id(x,y)^2}{t}}.
\end{align*}
b) For every $z_0\in M$ and every $R>0$ there exists a constant $C>0$ (which depends on the geometry of $M$ in a neighbourhood of $B(z_0,R)$) such that for all $(t,x,y)\in (0,R]\times B(z_0,R) \times B(z_0,R)$ one has 
\begin{align*}
&\left|\Id \log p^y(t,x)\right|\leq C(t^{-1/2}+t^{-1}\Id(x,y) ),\quad\text{where for every fixed $y\in M$ we have set}\\
&p^y:(0,\infty)\times M\longrightarrow (0,\infty),\>p^y(t,x):=p(t,x,y),
\end{align*}
and where here and in the sequel $\Id \log p^y(t,x):=\Id_x \log p^y(t,x)$, that is, the exterior differential of a function on space-time is always understood with respect to the space variable.
\end{Proposition}

\begin{proof} a) The localized heat kernel bounds that we have recorded in the appendix (cf. Section \ref{ssaa}), show the existence of constants $A_j>0$, $j=1,\dots,4$, that only depend on $m$ and a lower bound of $\mathrm{Ric}$ in a neighbourhood of $B(z_0,R)$, such that for all $(t,x,y)\in (0,R]\times B(z_0,R) \times B(z_0,R)$ one has 
\begin{align*}
&\mathrm{e}^{-A_1 t}\mu(B(x,\sqrt{t})^{-1/2}\mu(B(y,\sqrt{t}))^{-1/2}\mathrm{e}^{-A_2\f{\Id(x,y)^2}{t}}\\
&\leq p(t,x,y)\\
&\leq \mathrm{e}^{A_3t}\mu(B(x,\sqrt{t}))^{-1/2}\mu(B(y,\sqrt{t})^{-1/2}\mathrm{e}^{-A_4\f{\Id(x,y)^2}{t}}.
\end{align*} 
As $\sqrt{t}\leq\sqrt{R}<R+1$ and $B(x,2\sqrt{t})\subset B(z_0,4(R+1))$, $B(y,2\sqrt{t})\subset B(z_0,4(R+1))$, applying Bishop-Gromov\rq{}s volume estimates locally (cf. Section \ref{ssaa}) we can pick $A_5,A_6>0$  (that only depend on $m$ and a lower bound of $\mathrm{Ric}$ on say $B(z_0,4(R+1))$), such that
$$
\max\{\mu(B(x,\sqrt{t})),\mu(B(y,\sqrt{t}))\}\leq   A_5t^{m/2}\mathrm{e}^{A_6 \sqrt{t}}\leq A_5\mathrm{e}^{A_6 \sqrt{R}}t^{m/2},
$$
which yields the estimate
$$
p(t,x,y)\geq A_5\mathrm{e}^{-A_1 R}\mathrm{e}^{-A_6 \sqrt{R}}t^{-m/2}\mathrm{e}^{-A_2\f{\Id(x,y)^2}{t}}.
$$
On the other hand, using again $\sqrt{t}\leq\sqrt{R}<R+1$, $B(x,2\sqrt{t})\subset B(z_0,4(R+1))$, $B(y,2\sqrt{t})\subset B(z_0,4(R+1))$, and applying a local volume doubling inequality (cf. Section \ref{ssaa}), we can pick $A_7>0$  (that only depends on $m$ and a lower bound of $\mathrm{Ric}$ on $B(z_0,4(R+1))$) such that 
\begin{align*}
&\min\{\mu(B(x,\sqrt{t})),\mu(B(y,\sqrt{t}))\}\geq (R+1)^{-m} \mathrm{e}^{-A_7(R+1)} \inf_{a\in B(z_0,R) }\mu(B(a,R+1))t^{-m/2}\\
&=:A_8t^{-m/2}.
\end{align*}
Thus we have 
$$
p(t,x,y)\leq \mathrm{e}^{A_3R}A_8t^{-m/2}\mathrm{e}^{-A_4\f{\Id(x,y)^2}{t}},
$$
completing the proof.\\
b) We will use Arnaudon/Thalmaier\rq{}s estimate (Theorem \ref{ham}) as follows: Define $S:= t/2$, $u(s,z):=p^y(s+t/2,z)$ and let $D:=B(z_0,2R)$. We can pick finitely many $w_1,\dots w_l\in D$ such that $B(z_0,R)\subset \bigcup^l_{j=1} B\big(w_j,\Id(w_j,\partial D)/2\big)$. Then with the above choices Theorem \ref{ham} immediately implies
\begin{align*}
\left|\Id \log p^y(x,t)\right|^2&\leq 2\left(\f{2}{t}+\f{\pi^2(m+\beta m+7)}{\min_{j=1,\dots,l}\Id(w_j,\partial D)^2}+ \f{K}{4\beta}+K\right)\\
&\times\left(4+\log\f{\sup_{s\in [0,t/2], z\in\overline{D}} p(t/2+s,z,y)}{p(t,x,y)}\right)^2,
\end{align*}
where $-K\leq 0$ is any lower bound on the Ricci curvature on $D=B(z_0,2R)$, and $\beta>0$ can be chosen arbitrarily. Finally, by part a), we can find constants $c_j>0$ that only depend on $m$ and a lower bound of $\mathrm{Ric}$ in a neighbourhood of $D$, such that
$$
\log\f{\sup_{s\in [0,t/2], z\in\overline{D}} p(t/2+s,z,y)}{p(x,y,t)}\leq c_1+c_2\Id(x,y)^2/t,
$$
showing the inequality
$$
\left|\Id \log p^y(t,x)\right|\leq C\Big(t^{-1/2}+t^{-1}\Id(x,y) +t^{-1/2}\Id(x,y) +1\Big),
$$
which proves the claim (noting that $1\leq R/t$ and $t^{-1/2}\leq R^{-1/2} t^{-1})$.
\end{proof}

\begin{proof}[Proof of Theorem \ref{main}] Let $R_0>0$, $z_0\in M$ be arbitrary. We prove the claim for $x,y\in B(z_0,R_0)$. To this end, we fix an arbitrary smooth compactly supported $f:M\to\IR$.

The proof is divided into four parts:\vspace{2mm}

\emph{Claim 1: With $p^y:=p(\bullet,\bullet,y)$, for every $0\leq t<s<T$, and $A\in \IFF_t$ one has}
$$
\f{\Id}{\Id s} \mathbb{E}\left[  1_A f(X^{x,y,T}_s)  \right]=\mathbb{E}\left[ 1_A \Big( (1/2)\Delta f(X^{x,y,T}_s)+ \big(\Id \log p^y(T-s,X^{x,y,T}_s),\Id f (X^{x,y,T}_s)\big)\Big)\right].
$$

Proof of Claim 1: In principle we follow \cite{driver} here, up to the fact that we have to use the Markoff property of the bridge (which makes the calculation a little more complicated): Using the time-inhomogeneous Markoff property of $X^{x,y,T}$ and the defining relation of the pinned Wiener measure (note that $s-t<T-t$), we can calculate
\begin{align}\nn
& \mathbb{E}\left[  1_A f(X^{x,y,T}_s)  \right]=  \mathbb{E}\left[  1_A \int f(\omega(s-t))\Id \IP^{X^{x,y,T}_t,y,T-t}(\omega)  \right]\\\label{uam}
&=\mathbb{E}\left[  1_A\f{1}{p^y(T-t,X^{x,y,T}_t)} \int p^y(T-s,\omega(s-t))f(\omega(s-t))\Id \IP^{X^{x,y,T}_t,T-t}(\omega)  \right].
\end{align}
Let us define a smooth function
$$
\Psi:[0,T)\times M \longrightarrow \IR,\>\>\Psi_r(z):=p^y(T-r,z)f(z),\>\>(r,z)\in [0,T)\times M.
$$
Then, using that the heat kernel solves the heat equation and that $\Delta$ is formally self-adjoint (note here that $\Psi_s$ has a compact support in $M$), one can easily deduce
$$
\f{\partial}{\partial s}P_{s-t}\Psi_s(z)=P_{s-t}\left[\left(\f{\partial}{\partial s}+(1/2)\Delta\right) \Psi_s\right](z),
$$
an expression, which using the product rule for the Laplace-Beltrami operator and once more that the heat kernel solves the heat equation, is seen to be equal to
\begin{align*}
&=P_{s-t}\left[p^y(T-s,\bullet)(1/2)\Delta f+\big(\Id p^y(T-s,\bullet),\Id f \big)\right](z)\\
&=P_{s-t}\left[p^y(T-s,\bullet)(1/2)\Delta f+p^y(T-s,\bullet)\big(\Id \log p^y(T-s,\bullet),\Id f \big)\right](z).
\end{align*}
Thus, using (\ref{uam}) and using the defining relation of the Wiener measure twice,
\begin{align*}
&\f{\Id}{\Id s} \mathbb{E}\left[  1_A f(X^{x,y,T}_s)  \right]=\f{\Id}{\Id s}\mathbb{E}\left[  1_A\f{1}{p^y(T-t,X^{x,y,T}_t)} P_{s-t}\Psi_s(X^{x,y,T}_t)  \right]\\
&=\mathbb{E}\left[  1_A\f{1}{p^y(T-t,X^{x,y,T}_t)} \f{\partial}{\partial s}P_{s-t}\Psi_s(X^{x,y,T}_t)  \right]\\
&=\mathbb{E}\left[  1_A\f{1}{p^y(T-t,X^{x,y,T}_t)}\right.\\
&\>\>\>\>\>\>\times\left. P_{s-t}\left[p^y(T-s,\bullet)(1/2)\Delta f+p^y(T-s,\bullet)\big(\Id \log p^y(T-s,\bullet),\Id f \big)\right](X^{x,y,T}_t)  \right]\\
&=\mathbb{E}\left[  1_A\f{1}{p^y(T-t,X^{x,y,T}_t)} \int  p^y(T-s,\omega(s-t)) \right.\\
&\>\>\>\>\>\>\left.\times\Big((1/2)\Delta f(\omega(s-t))+\big(\Id \log p^y(T-s,\omega(s-t)),\Id f( \omega(s-t))\big)\Big) \Id\IP^{X^{x,y,T}_t,T-t}(\omega)  \right].
\end{align*}
In view of the defining relation of the pinned Wiener measure, the latter expression is equal to
\begin{align*}
&=\mathbb{E}\left[  1_A \int\Big((1/2)\Delta f(\omega(s-t))+\big(\Id \log p^y(T-s,\omega(s-t)),\Id f( \omega(s-t))\big)\Big)\right. \\
&\left.\>\>\>\>\>\>\>\>\>\times\Id\IP^{X^{x,y,T}_t,y,T-t}(\omega)  \right]
\end{align*}
and using the time-inhomogeneous Markoff property of $X^{x,y,T}$,
$$
=\mathbb{E}\left[  1_A \Big((1/2)\Delta f(X^{x,y,T}_s)+\big(\Id \log p^y(T-s,X^{x,y,T}_s),\Id f( X^{x,y,T}_s)\big)\Big)\right],
$$
which completes the proof of Claim 1.\vspace{2mm}

\emph{Claim 2: One has}
$$
\mathbb{E}\left[\int^T_0 \left|\Id \log p^y(T-r,X^{x,y,T}_r)\right|\left| \Id f(X^{x,y,T}_r) \right| \Id r\right]<\infty.
$$

Proof of Claim 2: Using the time reversal symmetry of the pinned Wiener measure we have (with $X$ the coordinate process on $\Omega_T(M)$)
\begin{align*}
&\mathbb{E}^{x,y,T}\left[\int^T_0 \left|\Id \log p^y(T-r,X_r)\right| \left| \Id f(X_r) \right|\Id r\right]\\
&=\mathbb{E}^{x,y,T}\left[\int^{T/2}_0 \left|\Id \log p^y(T-r,X_r)\right| \left| \Id f(X_r) \right|\Id r\right]\\
&+\mathbb{E}^{y,x,T}\left[\int^{T/2}_0 \left|\Id \log p^y(r,X_r)\right|\left| \Id f(X_r) \right| \Id r\right].
\end{align*}
In the first summand, the time variable $r$ remains uniformly away from the heat kernel singularity, which easily entails that this term is finite (as $f$ has a compact support). It remains to estimate the second summand. To this end, we pick $R>\max(R_0,T)$ large enough such that $ B(z_0,R)\supset\mathrm{supp}(f)$. Using the defining relation of the pinned Wiener measure, and of the Wiener measure, respectively, we have for every $0< r\leq T/2$,
\begin{align*}
&p(T,y,x)\mathbb{E}^{y,x,T}\left[ \left|\Id \log p^y(r,X_r)\right| \left| \Id f(X_r) \right|\right]\\
&=\mathbb{E}^{y,T}\left[p(T-r,X_r,x)  \left|\Id \log p^y(r,X_r)\right| \left| \Id f(X_r) \right| \right]\\
&= \int p(r,y,z)  p(T-r,z,x)  \left|\Id \log p^y(r,z)\right| \left| \Id f(z) \right| \Id\mu(z)\\
&=\int_{B(z_0,R)} p(r,y,z)  p(T-r,z,x)  \left|\Id \log p^y(r,z)\right| \left| \Id f(z) \right| \Id\mu(z)\\
&\leq \left\|\Id f\right\|_{\infty} \int_{B(z_0,R)} p(r,y,z)  p(T-r,z,x)  \left|\Id \log p^y(r,z)\right|  \Id\mu(z)\\
&\leq \Big\{\sup_{u\in [0,T/2],a,b\in B(z_0,R)}p(T-u,a,b)\Big\}\left\|\Id f\right\|_{\infty} \int_{B(z_0,R)} p(r,y,z)    \left|\Id \log p^y(r,z)\right|  \Id\mu(z)\\
&=:A \int_{B(z_0,R)} p(r,y,z)    \left|\Id \log p^y(r,z)\right|  \Id\mu(z).
\end{align*}
Next, using Proposition \ref{help}, we pick a constant $C>0$ such that for all $u\in (0,T]$, $x_1,x_2\in B(z_0,R)$ one has
$$
\left|\Id \log p(u,x_1,x_2)\right|\leq C  u^{-1/2}+C u^{-1}\Id(x_1,x_2)
$$
and
$$
p(u,x_1,x_2)\leq C  u^{-m/2} \mathrm{e}^{-C \f{\Id(x_1,x_2)^2}{u}},
$$
so that using $\int_M p(\alpha,v,w)\Id\mu(w)\leq 1$ for all $\alpha>0$, $v\in M$, 
\begin{align*}
&A^{-1}p(T,y,x)\mathbb{E}^{y,x,T}\left[ \left|\Id \log p^y(r,X_r)\right| \left| \Id f(X_r) \right|\right]\\
&\leq Cr^{-1/2}\int_{B(z_0,R)} p(r,y,z)     \Id\mu(z)+C r^{-1}\int_{B(z_0,R)} p(r,y,z) \Id(y,z)     \Id\mu(z)\\
&\leq Cr^{-1/2}+C^2  r^{-1-m/2}\int_{B(z_0,R)}  \mathrm{e}^{-C \f{\Id(z,y)^2}{r}} \Id(y,z)     \Id\mu(z).
\end{align*}
The first summand is integrable in $r$ from $0$ to $T/2$. For the second summand we proceed as follows: Covering $B(z_0,R)$ with finitely many balls centered in $y$ and having radius $<r_{\mathrm{inj}}(B(z_0,R))$, and using polar coordinates, we can estimate,
\begin{align*}
 & r^{-1-m/2}\int_{B(z_0,R)}  \mathrm{e}^{-C \f{\Id(z,y)^2}{r}} \Id(y,z)     \Id\mu(z)\leq c r^{-1-m/2}\int^{\infty}_0  \mathrm{e}^{-C \f{u^2}{r}} u   u^{m-1}  \Id u\\
&\leq d r^{-1-m/2+m/2+1/2}=dr^{-1/2},
\end{align*}
where $d>0$, $c>0$, which is again an integrable function of $r$ in $[0,T/2]$. Above we have made use of the Gaussian moments
$$
\int^{\infty}_0 \exp(-a u^2)u^m \Id u=C_m a^{-m/2-1/2},\quad a>0. 
$$
This completes the proof of Claim 2.\vspace{2mm}

\emph{Claim 3: The real-valued process
\begin{align*}
&Y:=X^{x,y,T,f}:[0,T]\times\Omega \longrightarrow \IR,\>\>Y_s:=f(X^{x,y,T}_s)-f(X^{x,y,T}_0)-(1/2)\int^s_0\Delta f(X^{x,y,T}_r)\Id r\\
&\>\>\>-\int^s_0 \big(\Id \log p^y(T-r,X^{x,y,T}_r),\Id f( X^{x,y,T}_r)\big)\Big)\Id r
\end{align*}
is a continuous $\IFF_*$-local martingale.} 

Proof of Claim 3: In view of Claim 2 and the fact that $f$ has a compact support, it is sufficient to prove that $Y|_{[0,T)}$ is an $\IFF_*$-martingale. To this end pick arbitrary $0\leq t<s<T$. It follows from applying $\int^t_s\cdots$ to the formula from Claim 1 that 
$$
\mathbb{E}\left[1_A (Y_t-Y_s)\right]=0\>\>\text{ for all $A\in\IFF_t$},
$$
which is equivalent to
$$
\mathbb{E}\left[Y_t|\IFF_s\right]=Y_s\>\>\text{ $\mathbb{P}$-a.s.}
$$

\emph{Claim 4: $f(X^{x,y,T})$ is a continuous $\IFF_*$ semimartingale}.

Proof of Claim 4: We have for all $s\in [0,T]$, $\IP$-a.s.
\begin{align*}
&f(X^{x,y,T}_s)=Y_s\\
&\>\>\>\>\>\>\>\>\>\>\>\>\>+f(X^{x,y,T}_0)+(1/2)\int^s_0\Delta f(X^{x,y,T}_r)\Id r+\int^s_0 \big(\Id \log p^y(T-r,X^{x,y,T}_r),\Id f( X^{x,y,T}_r)\big)\Big)\Id r,
\end{align*}
Thus by Claim 2 and Claim 3 this is a sum of a continuous local martingale and a continuous adapted process with paths having a finite variation. This completes the proof. 
\end{proof}

By standard results on manifold-valued continuous semimartingales \cite{hack}, we have:

\begin{Corollary}\label{sjoa} Given a filtered probability space $(\Omega,\IFF,(\IFF_t)_{t\in [0,T]},\IP)$ which satisfies the usual assumptions, and an $\IFF_*:=(\IFF_t)_{t\in T}$-Brownian bridge $X^{x,y,T}$ from $x$ to $y$ with terminal time $T$. Assume further that we are given a smooth $G$-principal bundle $\pi:P\to M$ together with a connection form $\alpha\in\Omega^1_{\ICC}(P,\mathrm{Lie}(G))$. Then for every $\IFF_0$-measurable random variable $u:\Omega\to P$ with $\pi(u)=x$ $\mathbb{P}$-a.s., there exists a unique (up to indistinguishability) $\alpha$-horizontal $\IFF_*$-lift $U:[0,T]\times\Omega\to P$ with $U_0=u$ $\mathbb{P}$-a.s., that is, $U$ is the uniquely determined continuous $\IFF_*$-semimartingale $U:[0,T]\times\Omega\to P$ which satisfies the following properties:
\begin{itemize}
\item $U_0=u$ $\mathbb{P}$-a.s.
\item $\pi(U_t)=X^{x,y,T}_t$ $\mathbb{P}$-a.s., for all $t\in [0,T]$
\item $\int^t_0\alpha(\underline{\Id} U_s)=0$ (Stratonovic line integral of $\alpha$ along $U$) $\mathbb{P}$-a.s., for all $t\in [0,T]$.
\end{itemize}
\end{Corollary}

\appendix 

\section{Localized heat kernel and volume bounds}\label{ssaa}

For the convenience of the reader we record here some facts on heat-kernels and volumes on geodesically complete Riemannian manifolds. Let $M$ be a geodesically complete connected smooth Riemannian $m$-manifold, with $p(t,x,y)$, $\mu$, and $B(x,r)$ as above. There hold the following facts (cf. Theorem 6.1 and inequality (1), (2) in \cite{saloff}):\\

(i) (Localized heat kernel bounds): \emph{For every $r>0$, $x\rq{}\in M$, there exist constants $A_j>0$ which only depend on $m$ and a lower bound of $\mathrm{Ric}$ in $B(x\rq{},2r)$, such that for all $(t,x,y)\in (0,r^2)\times B(x\rq{},r)\times B(x\rq{},r)$ one has}
\begin{align*}
&\mathrm{e}^{-A_1 t}\mu(B(x,\sqrt{t}))^{-1/2}\mu(B(y,\sqrt{t}))^{-1/2}\mathrm{e}^{-A_2\f{\Id(x,y)^2}{t}}\\
&\leq p(t,x,y)\\
&\leq \mathrm{e}^{A_3t}\mu(B(x,\sqrt{t}))^{-1/2}\mu(B(y,\sqrt{t}))^{-1/2}\mathrm{e}^{-A_4\f{\Id(x,y)^2}{t}}.
\end{align*} 

(ii) (Cheeger-Gromov estimate) \emph{For every $r>0$, $ x\in M$, one has
$$
\mu(B(x,s))\leq |\mathbb{S}^m|s^{m}\mathrm{e}^{\sqrt{(m-1)K}s}\quad\text{ for all $0<s<2r$,}
$$
where $K\geq 0$ is any lower bound of $\mathrm{Ric}\geq -K$ in $B(x ,2r)$.}\vspace{2mm}

(iii) (Local volume doubling property) \emph{For every $r>0$, $ x\in M$, one has
$$
\mu(B(x,s))\leq \mu(B(x,s\rq{})) (s/s\rq{})^m \mathrm{e}^{\sqrt{(m-1)K}s}\quad\text{ for all $0<s\rq{}<s<2r$,}
$$
where again $K\geq 0$ is any number such that $\mathrm{Ric}\geq -K$ in $B(x ,2r)$.}

\section{Arnaudon-Thalmaier\rq{}s local gradient estimate}

We record here the following localized version of Hamilton\rq{}s gradient estimate, which is by Arnaudon and Thalmaier. Let now $M$ be an arbitrary connected smooth Riemannian $m$-manifold.

\begin{Theorem}\label{ham} Let $D\subset M$ be an open relatively compact subset, let $S>0$, and let a continuous function
$$
u: [0,S]\times \overline{D}\longrightarrow (0,\infty)
$$
be given, which is a smooth solution of
$$
\f{\partial }{\partial s} u(s,z) =\f{1}{2}\Delta u(s,z)\>\>\text{ in $(s,z)\in [0,S]\times D$.}
$$
Then for all $K\geq 0$ with $\mathrm{Ric}|_D\geq- K$, all $\beta>0$, $w\in D$, and all $z\in B\big(w,\Id(w,\partial D)/2\big)$ one has the gradient bound
\begin{align*}
\left|\Id \log  u(z,S)\right|^2&\leq 2\left(\f{1}{S}+\f{\pi^2(m+\beta m+7)}{\Id(w,\partial D)^2}+ \f{K}{4\beta}+K\right)\left(4+\log\f{\sup_{[0,S]\times \overline{D}} u}{u(z,S)}\right)^2.
\end{align*}
\end{Theorem}

This is an immediate consequence of Theorem 7.1 in \cite{thal}, where an entire probabilistic proof has been given. We refer the reader also to \cite{wang} for analgous techniques.


\begin{thebibliography}{99}

\bibitem{aida} Aida, S.: \emph{Logarithmic derivatives of heat kernels and logarithmic Sobolev inequalities with unbounded diffusion coefficients on loop spaces.} J. Funct. Anal. 174 (2000), no. 2, 430--477.


\bibitem{thal} Arnaudon, M. \& Thalmaier, A.:\emph{ Li-Yau type gradient estimates and Harnack inequalities by stochastic analysis.} Probabilistic approach to geometry, 29--48, Adv. Stud. Pure Math., 57, Math. Soc. Japan, Tokyo, 2010.


\bibitem{baer} Bär, C. \& Pfäffle, F.: \emph{Wiener measures on Riemannian manifolds and the Feynman-Kac formula.} Mat. Contemp. 40 (2011), 37--90.


\bibitem{bismut} Bismut, J.-M.. \emph{Large deviations and the Malliavin calculus.} Progress in Mathematics, 45. Birkhäuser Boston, Inc., Boston, MA, 1984. 

\bibitem{bisat} Bismut, J.-M.: \emph{The Atiyah-Singer theorems: a probabilistic approach. I. The index theorem.} J. Funct. Anal. 57 (1984), no. 1, 56--99. 




\bibitem{driver} Driver, B.K.: \emph{A Cameron-Martin type quasi-invariance theorem for pinned Brownian motion on a compact Riemannian manifold.} Trans. Amer. Math. Soc. 342 (1994), no. 1, 375--395.






\bibitem{buch} Grigor'yan, A.: {\it Heat kernel and analysis on manifolds.} AMS/IP Studies in Advanced Mathematics, 47. American Mathematical Society, Providence, RI; International Press, Boston, MA, 2009.














\bibitem{G2} Güneysu, B.: {\it On generalized Schrödinger semigroups.} J. Funct. Anal. 262 (2012), 4639--4674.

\bibitem{G1} Güneysu, B.: \emph{Multiplicative matrix-valued functionals and the continuity properties of semigroups corresponding to partial differential operators with matrix-valued coefficients.} J. Math. Anal. Appl. 380 (2011), no. 2, 709--725.


%
%

\bibitem{hack} Hackenbroch, W. \& Thalmaier, A.: \emph{Stochastische Analysis}. Teubner Verlag, Stuttgart, 1994.

%


\bibitem{Hsu} Hsu, E.: {\it Stochastic Analysis on Manifolds.} AMS, 2002.

\bibitem{hsu2} Hsu, E. \emph{Brownian bridges on Riemannian manifolds.} Probab. Theory Related Fields 84 (1990), no. 1, 103--118.

\bibitem{ikeda} Ikeda, N. \& Watanabe, S.: \emph{ Stochastic differential equations and diffusion processes.} North-Holland Mathematical Library, 24. North-Holland Publishing Co., Amsterdam-New York; Kodansha, Ltd., Tokyo, 1981.








 




%
%




%





%





\bibitem{saloff} Saloff-Coste, L.: \emph{Uniformly elliptic operators on Riemannian manifolds.} J. Differential Geom. 36 (1992), no. 2, 417--450. 


%



%

\bibitem{sznit} Sznitman, A.-S.: \emph{Brownian motion, obstacles and random media. Springer Monographs in Mathematics.} Springer-Verlag, Berlin, 1998. 


\bibitem{wang} Thalmaier, A. \& Wang, F.-Y.: \emph{Gradient estimates for harmonic functions on regular domains in Riemannian manifolds.} J. Funct. Anal. 155 (1998), 109--124.


\end{thebibliography}
\end{document}